\documentclass[12pt]{amsart}
\usepackage{verbatim,amscd,amsmath,amssymb}

\newtheorem{theorem}{Theorem}[section]

\newtheorem{lemma}[theorem]{Lemma}

\newtheorem{question}[theorem]{Question}
\newtheorem{definition}[theorem]{Definition}

\newtheorem{corollary}[theorem]{Corollary}
\theoremstyle{plain}
\numberwithin{equation}{theorem}

\theoremstyle{remark}
\newtheorem{remark}[theorem]{Remark}
\newtheorem{example}[theorem]{Example}

\newcommand{\Z}{{\mathbb Z}}

\newcommand{\Pl}{{\mathbb P}}

\newcommand{\lra}{\longrightarrow}

\title[Projection of the diagonal subvariety]{A note on the projection of the diagonal subvariety}
\author{Jorge Pineiro}

\address{Jorge Pineiro: Department of Mathematics and Computer Science.
Bronx Community College of CUNY.
2155 University Ave.
Bronx, NY 10453}

\email{jorge.pineiro@bcc.cuny.edu}


\begin{document}
\begin{abstract}
The diagonal subvariety on a product of two CM elliptic curves, is presented as an example of a dimension one subvariety, that is pre-periodic only if the respective projection on the product of two projective lines is also pre-periodic.

\end{abstract}
\maketitle

\section{Introduction}
 In \cite{DMMC} the authors presented a contradiction to the Dynamical Manin-Mumford conjecture using product of CM-elliptic curves.  They considered maps $[\omega] \times [\omega'] : E \times E \lra E \times E$ on the product $E \times E$, where $E$ is an elliptic with complex multiplication by $\omega, \omega'$, and studied the orbit of the diagonal subvariety $\Delta=\{(P,P): P \in E)\}$.  In order to state the result obtained in \cite{DMMC} we define a subvariety $Y \subset X$ to be pre-periodic for a map $\varphi$ if it has a finite forward orbit $\{ Y,\varphi_*Y,\dots\}$ under $\varphi$.  The characterization of when is $\Delta$ pre-periodic for the map $\varphi=[\omega] \times [\omega']$ is given by (c.f. lemma \ref{diaginEE}):
 \begin{lemma}
$\Delta$ is pre-periodic for $[\omega] \times [\omega'] : E \times E \lra E  \times E$ if and only if $\omega/\omega' $ is a root of unity.
\end{lemma}
 Similar properties can be found for the diagonal subvariety $\Delta'$ inside the product of projective lines under the action of Latt\'{e}s maps.  Any map $[\omega] : E \lra E$ on an elliptic curve $E$ can be projected to a map $\varphi_{[\omega]} : \Pl^1 \lra \Pl^1$, when we mod out by the hyperelliptic involution $\pi : E \lra \Pl^1$.  The analogous result characterizing when is the orbit of the diagonal $\Delta'=\{(P,P): P \in \Pl^1  \}$ finite under the action of the product of two Latt\'{e}s maps reads (c.f. lemma \ref{diaginP1P1}):
 \begin{lemma}
$\Delta'$ is preperiodic for $\varphi_{[\omega']}\times \varphi_{[\omega']} : \Pl^1  \times \Pl^1 \lra \Pl^1 \times  \Pl^1$ if and only if $\omega/\omega' $ is a root $\pm 1$.
\end{lemma}
 Combining these two results together we have an example of subvarieties $\Delta,\Delta'$ satisfying:
 \begin{theorem} $\Delta$ is pre-periodic for $[\omega] \times [\omega']$ if and only if $\pi_*(\Delta)=\Delta'$ is pre-periodic for $\varphi_{[\omega]} \times \varphi_{[\omega']}$.
 \end{theorem}
 Note that the maps $[\omega] \times [\omega']: E \times E \lra E \times E$, $\varphi_{[\omega]} \times \varphi_{[\omega']}: \Pl^1 \times \Pl^1 \lra \Pl^1 \times \Pl^1$ fit into a commutative diagram:
 \[
\begin{CD}
 \Delta \subset E \times E            @>[\omega] \times [\omega']>>   E \times E \\
    @V (\pi,\pi) VV            @V (\pi,\pi) VV          \\
  \pi(\Delta)=\Delta' \subset \Pl^1 \times \Pl^1   @>\varphi_{[\omega]} \times \varphi_{[\omega']}>> \Pl^1 \times \Pl^1 .\\
\end{CD}\vspace{3mm}
\]
In the next section we will put the above example in the context of commutative diagram of varieties.  We will show that the orbits of $\Delta$ and $\Delta'$ for some values of $\omega$ and $\omega'$ does not have the same length.

 \section{Projection of subvarieties in a commutative diagram} Let $X$ be a projective algebraic variety.  An algebraic dynamical system is a self-map $\varphi : X \lra X$, defined on the algebraic variety $X$.  A subvariety $Y \subset X$ is said to be pre-periodic for $\varphi$ if the forward orbit under $\varphi$ is finite, more precisely

  \begin{definition}The subvariety $Y \varsubsetneq X$ is pre-periodic for $\varphi : X \lra X$ if there are natural numbers $n,k$, with $k>0$, such that $\varphi_*^{n+k}(Y)=\varphi_*^n(Y)$.  In this situation we say that $Y$ is pre-periodic with pair $(n,k)$. \end{definition}
  \begin{example}If $E$ is an elliptic curve, $P$ is a torsion point with $[k]P=0$ and we take the point $Q$ such that $[n]Q=P$, the point $Q$ will be pre-periodic for the map $[n]:E \lra E$ with pair $(n,k)$.  \end{example}
   Suppose now that we have algebraic dynamical systems $\varphi, \tilde{\varphi}$ that fit into a commutative diagram as follows:
\[
\begin{CD}
 X            @>\varphi>>   X  \\
@V \pi VV        @VV \pi V    \\
  \tilde{X}    @>\tilde{\varphi}>> \tilde{X}\\
\end{CD} \vspace{2mm}
\]
The first thing to note is the following:
\begin{remark}
 If a subvariety $Y \subset X$ is pre-periodic for $\varphi$ then the projection $\pi_*Y$ is pre-periodic for $\tilde{\varphi}$.  Indeed, $Y$ pre-periodic implies the existence of natural numbers $n,k (k>0)$ such that, $$ \varphi_*^{n+k}(Y)=\varphi_*^n(Y) \Rightarrow \pi_*(\varphi_*^{n+k}(Y))=\pi_*(\varphi_*^n(Y)) \Rightarrow \varphi_*^{n+k}(\pi_*Y)=\varphi_*^n(\pi_*Y),$$
 and therefore $\pi_*Y$ will also be pre-periodic with the same pair $(n,k)$.
\end{remark}
We are interested then in studying the following question:
\begin{question} \label{Q1}
Is it true that $Y \subsetneq X$ is pre-periodic for $\varphi$ if and only if the projection $\pi_* Y \subsetneq \tilde{X}$ is pre-periodic for $\tilde{\varphi}$? In case of a positive answer, does it happens with the same pair $(n,k)$?
\end{question}
\begin{remark}
For subvarieties of dimension zero $Y=P$ with reduced structure.  The theory of canonical height functions developed by Silverman and other authors in \cite{silvermancanonicalheights} or \cite{intro-diophantine} for polarized dynamical systems, allows to completely characterize the pre-periodic points as points with height $\hat{h}_{\varphi}(P)=0$.  The first part of the question \ref{Q1} is settled this way.  For higher dimensional subvarieties, the results of Ghioca, Tucker and Zhang \cite{DMMC} show how the vanishing of height is not a sufficient condition for a subvariety to be pre-periodic.

\end{remark}
\subsection{The diagonal subvariety inside the product of elliptic curves} We consider the following situation of a subvariety of dimension one.  Let $E$ be an elliptic curve with complex multiplication by a ring $R$ and $\omega, \omega' \in R$.  Let $\pi: E \rightarrow \Pl^1$ be the map arising when we mod out by the hyperelliptic involution.  We want to study Question \ref{Q1} for the diagram:
\[
\begin{CD}
 \Delta \subset E \times E            @>[\omega] \times [\omega']>>   E \times E \\
    @V (\pi,\pi) VV            @V (\pi,\pi) VV          \\
  \pi(\Delta)=\Delta' \subset \Pl^1 \times \Pl^1   @>\varphi_{[\omega]} \times \varphi_{[\omega']}>> \Pl^1 \times \Pl^1, \\
\end{CD}\vspace{3mm}
\]
where $\Delta$ and $\Delta'$ are respectively the diagonal subvarieties:
$$\Delta=\{(P,P): P \in E \}, \qquad \Delta'=\{(P,P): P \in \Pl^1 \}$$
The result is the following:
\begin{theorem} \label{Equivalence-of-periodic}
Let $\Delta \subsetneq E \times E$ and $\Delta' \subsetneq \Pl^1 \times \Pl^1$ denote respectively the diagonal subvarieties.  Suppose that $\omega, \omega' \in R$, then $\Delta$ is pre-periodic for $[\omega] \times [\omega'] : E \times E \lra E  \times E$ if and only if $\Delta'$ is preperiodic for $\varphi_{[\omega]}\times \varphi_{[\omega']} : \Pl^1  \times \Pl^1 \lra \Pl^1 \times  \Pl^1$.
\end{theorem}
It is a consequence of the two lemmas:

\begin{lemma} \label{diaginEE}
$\Delta$ is preperiodic for $[\omega] \times [\omega'] : E \times E \lra E  \times E$ if and only if $\omega/\omega' $ is a root of unity.
\end{lemma}
\begin{proof}
We reproduce the proof of Ghioca and Tucker.  Suppose that $([\omega]^{n+k}, [\omega']^{n+k})(\Delta)=([\omega]^{n}, [\omega']^{n})(\Delta)$ for some $n,k>0$.  Consider a non-torsion point $P \in E$, then there exist $Q \in E$ also non-torsion such that $([\omega]^{n+k}, [\omega']^{n+k})(P,P)=([\omega]^{n}, [\omega']^{n})(Q,Q)$.  But then $[\omega]^{n+k}(P)=[\omega]^{n}(Q)$ and $[\omega']^{n+k}(P)=[\omega']^{n}(Q)$ or equivalently $[\omega]^{n}([\omega]^{k}(P)-Q)=0$ and $[\omega']^{n}([\omega']^{k}(P)-Q)=0$.  These last two equations are saying that there are
torsion points $P_1,P_2$ such that $[\omega]^{k}(P)-Q=P_1$ and $[\omega']^{k}(P)-Q=P_2$ and therefore $[\omega]^{k}(P)-[\omega']^{k}(P)$ will also be a torsion point, and that cannot be for $P$ non-torsion unless $[\omega]^{k}-[\omega']^{k}=0$ and therefore $\omega/\omega'$ is a root of unity.  Conversely, suppose that $[\omega]^{k}=[\omega']^{k}$, then $([\omega]^{n+k},[\omega']^{n+k})(P,P)=([\omega]^{n},[\omega']^{n})([\omega]^{k}(P),[\omega]^k(P))$ and because $[\omega]^k$ is surjective $([\omega]^{n+k}, [\omega']^{n+k})(\Delta) = ([\omega]^{n}, [\omega']^{n})(\Delta)$. \end{proof}
\begin{lemma} \label{diaginP1P1}
$\Delta'$ is preperiodic for $\varphi_{[\omega']}\times \varphi_{[\omega']} : \Pl^1  \times \Pl^1 \lra \Pl^1 \times  \Pl^1$ if and only if $\omega/\omega' $ is a root $\pm 1$.
\end{lemma}
\begin{proof}
The proof is analogous to the case of $\Delta$.  Suppose that we have $(\varphi_{[\omega]}, \varphi_{[\omega']})^{n+k}(\Delta')=(\varphi_{[\omega]}, \varphi_{[\omega']})^n(\Delta')$ for some $n \geq 0$ and $k>0$.  Then $(\pi,\pi)([\omega]^{n+k}, [\omega']^{n+k})(\Delta)=(\pi,\pi)([\omega]^{n}, [\omega']^{n})(\Delta)$ and for each $P \in E$ there will be $Q \in E$ with $([\omega]^{n+k}, [\omega']^{n+k})(P,P)=([\omega]^{n}, [\omega']^{n})(Q,\pm Q)$.  But then $[\omega]^{n+k}(P)=[\omega]^{n}(Q)$ and $[\omega']^{n+k}(P)=[\omega']^{n}(\pm Q)$ or equivalently $[\omega]^{n}([\omega]^{k}(P)-Q)=0$ and $[\omega']^{n}([\omega']^{k}(P)\mp Q)=0$.  These last two equations are saying that there are
torsion points $P_1,P_2$ such that $[\omega]^{k}(P)-Q=P_1$ and $[\omega']^{k}(P)\mp Q=P_2$ and therefore $P_1 \mp P_2=[\omega]^{k}(P) \mp [\omega']^{k}(P)=([\omega]^{k} \mp [\omega']^{k})(P)$ will also be a torsion point, and that cannot be if we choose $P$ non-torsion unless $[\omega]^{k} \mp [\omega']^{k}=0$ and $(\omega/\omega')^k= \pm 1$. \end{proof}
\begin{corollary}If $(\omega/\omega')^k=-1$, $\Delta'$ is pre-periodic with pair $(n,k)$ but $\Delta$ is not pre-periodic with the same pair.  If $(\omega/\omega')^k=1$, both $\Delta$ and $\Delta'$ are pre-periodic with the same pair $(n,k)$.\end{corollary}

\begin{example} The elliptic curve $E : y^2=x^3+x$
admits complex multiplication by $R=\Z[i]$.  When we take $\omega=2+i$ and $\omega'=1-2i$ in $R$ we have $\omega/\omega'=-i$, therefore in the diagram:
\[
\begin{CD}
 \Delta \subset E \times E            @>[2+i] \times [1-2i]>>   E \times E \\
    @V (\pi,\pi) VV            @V (\pi,\pi) VV          \\
  \pi(\Delta)=\Delta' \subset \Pl^1 \times \Pl^1   @>\varphi_{[2+i]} \times \varphi_{[1-2i]}>> \Pl^1 \times \Pl^1 \\
\end{CD}\vspace{3mm}
\]
we will have that $\Delta$ is pre-periodic for $[2+i]\times [1-2i]$ with pair $(0,4)$ (but not $(0,2)$), while $\Delta'$ is pre-periodic for $\varphi_{[2+i]} \times \varphi_{[1-2i]}$ with pair $(0,2)$.  We refer to \cite{Pineiro-2009} for the explicit computation of the maps $\varphi_{\omega},\varphi_{\omega'}$.
\end{example}

\begin{example}
The curve $E : y^2=x^3+1$
admits complex multiplication by the ring $R=\Z[\rho]$ where
$\rho=(\sqrt{-3}-1)/2$.  For $\omega=\sqrt{3}\rho$ and $\omega'=\sqrt{3}$ we will get $(\omega/\omega')^3=1$ and $\Delta, \Delta'$ will be pre-periodic with the same pair $(n,3)$.  Again we refer to \cite{Pineiro-2009} for explicit computations of the maps $\varphi_{\omega},\varphi_{\omega'}$.  \end{example}

\providecommand{\bysame}{\leavevmode\hbox to3em{\hrulefill}\thinspace}

\end{document}